\documentclass[leqno,11pt]{amsart}
\usepackage{amssymb,amsmath,latexsym,amssymb,amsfonts,amsbsy,amsthm,esint}
\usepackage{tikz}
\usepackage{color}
\usepackage{graphicx}
\usepackage{hyperref}

\def\jump#1{{[\hspace{-2pt}[#1]\hspace{-2pt}]}}

\let\pa=\partial
\let\ka=\kappa
\let\al=\alpha
\let\b=\beta
\let\g=\gamma
\let\d=\delta

\let\r=\rho
\let\s=\sigma
\let\f=\frac

\let\om=\omega
\let\G= \Gamma

\let\Om=\Omega

\let\e=\varepsilon
\let\pa=\partial

\let\ri=\rightarrow
\let\na=\nabla

\newcommand{\BR}{\mathbb{R}}
\newcommand{\bz}{\mathbf{z}}

\def\tom{\widetilde{\om}}

\newcommand{\beq}{\begin{equation}}
\newcommand{\eeq}{\end{equation}}
\newcommand{\beqo}{\begin{equation*}}
\newcommand{\eeqo}{\end{equation*}}
\newcommand{\ben}{\begin{eqnarray}}
\newcommand{\een}{\end{eqnarray}}
\newcommand{\beno}{\begin{eqnarray*}}
\newcommand{\eeno}{\end{eqnarray*}}

\newtheorem{theorem}{Theorem}[section]

\newtheorem{proposition}[theorem]{Proposition}

\theoremstyle{remark}

\title[A unified approach towards the impossibility of finite time vanishing depth]{A unified approach towards the impossibility of finite time vanishing depth for incompressible free boundary flows}

\author[Z. Geng]{Zhiyuan Geng}
\email{zgeng@bcamath.org}
\address{BCAM--Basque Center for Applied Mathematics, Mazarredo 14, E48009
Bilbao, Basque Country, Spain}

\author[R.Granero-Belinch\'{o}n]{Rafael Granero-Belinch\'{o}n}
\email{rafael.granero@unican.es}
\address{Departamento  de  Matem\'aticas,  Estad\'istica  y  Computaci\'on, Universidad  de Cantabria, Santander,  Spain}

\begin{document}

\begin{abstract}
In this paper we study the motion of an internal water wave and an internal wave in a porous medium. For these problems we establish that, if the free boundary and, in the case of the Euler equations, also the tangential velocity at the interface are sufficiently smooth, the depth cannot vanish in finite time. This results holds regardless of gravity and surface tension effects or, if applicable, the stratification in multiphase flows.
\end{abstract}

\maketitle

\section{Introduction}
This paper studies the motion of an incompressible fluid with a free boundary in presence of an impervious bottom. Example of such a flow are the sea near the shoreline or the case of a bounded porous aquifer. For these problems we establish a number of conditions on the interface and the tangential velocity of the interface such that if they hold, the depth cannot vanish (see Figure \ref{fig1}) in finite time. These results hold regardless of the stratification, gravity or surface tension effects. Let us also remark that, in the case of two phase Euler, a similar result was proved using Lagrangian coordinates in a very interesting paper by Daniel Coutand in \cite{coutand2019finite} (see also the related results \cite{CoSh2016,FeIoLi2016}). However, our work presents a different and unified approach to this question. Our approach is based on a careful study of the contour equation formulation.

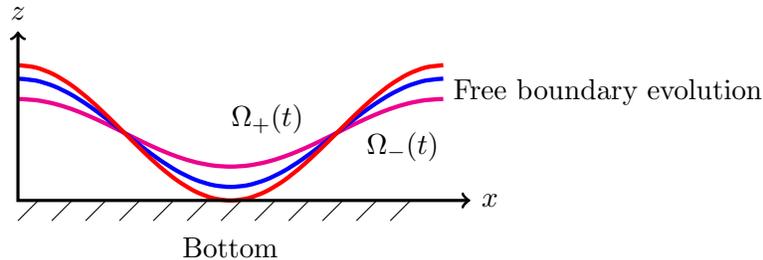
\begin{figure}[h]\label{fig1}
\begin{center}
\begin{tikzpicture}[domain=0:2*pi, scale=0.9]
\draw[ultra thick, smooth, color=magenta] plot (\x,{1+0.5*cos(\x r)});
\draw[ultra thick, smooth, color=blue] plot (\x,{1+0.8*cos(\x r)});
\draw[ultra thick, smooth, color=red] plot (\x,{1+cos(\x r)});
\draw[very thick,<->] (2*pi+0.4,0) node[right] {$x$} -- (0,0) -- (0,2.5) node[above] {$z$};
\node[right] at (3,1.2) {$\Omega_+(t)$};
\node[right] at (5,0.8) {$\Omega_-(t)$};
\node[right] at (2*pi,1.6) {Free boundary evolution};
\draw[-] (0,-0.3) -- (0.3, 0);
\draw[-] (0.5,-0.3) -- +(0.3, 0.3);
\draw[-] (1,-0.3) -- +(0.3, 0.3);
\draw[-] (1.5,-0.3) -- +(0.3, 0.3);
\draw[-] (2,-0.3) -- +(0.3, 0.3);
\draw[-] (2.5,-0.3) -- +(0.3, 0.3);
\draw[-] (3,-0.3) -- +(0.3, 0.3);
\draw[-] (3.5,-0.3) -- +(0.3, 0.3);
\draw[-] (4,-0.3) -- +(0.3, 0.3);
\draw[-] (4.5,-0.3) -- +(0.3, 0.3);
\draw[-] (5,-0.3) -- +(0.3, 0.3);
\draw[-] (5.5,-0.3) -- +(0.3, 0.3);
\draw (pi,-0.7) node {Bottom};
\end{tikzpicture}
\end{center}
\caption{The situation that is discarded by our results.}
\end{figure}

The study of incompressible fluids with a free surface is a long-standing reasearch problem for physicists, engineers and mathematicians since the XVIIIth century. In particular, the possible formation of finite time singularities where certain geometrical quantity of the free boundary blows up have been studied since the end of the XVIIIth century \cite{mariotte1700traite}. There are a number of well-known physical situations that we can fit in this framework. For instance, we can consider drop formation \cite{eggers1997nonlinear,oron1997long,moseler2000formation} and pinch-off singularities \cite{constantin2018singularity,alvarez2009pinch}, corner singularities \cite{liu2019local,liu2021search} or wetting \cite{bae2021singularity,camassa2019hydrodynamic,camassa2019singularity} to give just some examples.

In this paper we consider two different problems. On the one hand we consider the case of an internal wave in a porous medium with an impervious bottom.
\begin{equation}\label{systemMuskat}
\begin{cases}
\mu_{\pm}u_\pm =-\na p_\pm-\r_\pm(0,1)^Tg &\text{ in } \Om_\pm(t),\\
\na\cdot u_\pm=0 &\text{ in } \Om_\pm(t),\\
p_+-p_-=\g\ka &\text{ on } \G(t),\\
u_-\cdot (0,1)^T=0 &\text{ on } \{y=0\},\\
u_-\cdot (\pa_\al\bz)^\bot =u_+\cdot (\pa_\al\bz)^\bot &\text{ on } \G(t).
\end{cases}
\end{equation}
Here $g>0$ is the acceleration due to gravity force, $\gamma$ is the surface tension strength and $\ka$ is the curvature of the free surface $\G(t)$, which, for a curve, is given by
\beqo
\ka=\f{\pa_\al z_1\pa_\al^2 z_2-\pa_\al z_2\pa_\al^2z_1}{\left[((\pa_\al z_1)^2)+(\pa_\al z_2)^2\right]^{3/2}}.
\eeqo
Similarly, $p_\pm$ is the pressure, $\mu_\pm$ and $\r_\pm$ are the viscosity and density of the fluids involved. This problem is known in the literature as the (confined) two-phase Muskat problem \cite{cordoba2007contour,cordoba2014confined,granero2020growth,gancedo2020surface}.

On the other hand, we consider the motion of a water wave bounded below by a fixed impervious bottom. In this second situation the appropriate PDE's are
\begin{equation}\label{systemEuler2}
\begin{cases}
\rho_\pm(\partial_t u_\pm+(u_\pm \cdot\nabla)u_\pm )=-\na p_\pm-\r_\pm(0,1)^Tg &\text{ in } \Om_\pm(t),\\
\na\cdot u_\pm=0 &\text{ in } \Om_\pm(t),\\
\na\times u_\pm=0 &\text{ in } \Om_\pm(t),\\
p_+-p_-=\g\ka &\text{ on } \G(t),\\
u_-\cdot (0,1)^T=0 &\text{ on } \{y=0\},\\
u_-\cdot (\pa_\al\bz)^\bot =u_+\cdot (\pa_\al\bz)^\bot &\text{ on } \G(t).
\end{cases}
\end{equation}

This situation has been studied both theoretically and numerically by many different authors (see for instance \cite{ambrose2021numerical,lannes2005well,lannes2013water}). In the case of two-phase Euler, the impossibility of vanishing depth for smooth enough initial data was obtained by Daniel Coutand in \cite[Theorem 2]{coutand2019finite}. In his results, the author established that, if
$$
\|\partial_\alpha \tau\|_{L^\infty(\BR)}+\int_0^T\|\nabla u_-(s)\|_{L^\infty(\Omega_-(s))}ds<\infty,
$$
where $\tau$ is the tangent vector to $\Gamma(t)$, then the interface cannot reach the impervious bottom.

In both cases, we assume that the fluids lie above the $x-$axis, \emph{i.e.}
$$
\Omega_+(t)\cup\Omega_-(t)=\{(x,y,t)\in \BR^2_+\times \BR_+\},
$$
where $\BR^2_+$ is the upper half plane $\{(x,y)\in \BR^2: \; y>0\}$, and $\BR_+:=\{t\in\BR:\; t>0\},$ the interface is described as
\beq\label{par-gamma}
\G(t)=\{\bz(\al,t)=(z_1(\al,t),z_2(\al,t)), \al\in\BR \},
\eeq
for some functions $z_i: \BR\times\BR_+\ri\BR, i=1,2$,
and the impervious bottom is given by
\beq\label{par-bott}
\G_{\text{bot}}=\{(\al,0), \al\in\BR \}.
\eeq
Then, both problems can be written in terms of nonlinear and nonlocal evolution equations \cite{cordoba2010interface,cordoba2011interface}. Namely,
\beq\label{eq for dt z}
\pa_t \bz(\al)= \f1{2\pi}p.v.\int_\BR \bigg[ \f{(\bz(\al)-\bz(s))^\bot}{|\bz(\al)-\bz(s)|^2}-\f{(\bz(\al)-(z_1(s),-z_2(s)))^\bot}{|\bz(\al)-(z_1(s),-z_2(s))|^2} \bigg]\tom(s)\,ds+c(\al,t)\pa_\al \bz,
\eeq
where $\tom$ denotes the strength of the vorticity (or equivalently, the difference between the tangential components of the velocities \cite{aurther2019rigorous}) and $c$ is an arbitrary function that reflects the invariance of the problem by reparametrization of the free surface. The difference between the case of water waves and the case of the Muskat problem relies on the unknown $\tom$. In the case of the Muskat problem $\tom$ in general satisfies a nonlinear integral equation \cite{cordoba2011interface}, while in the case of water waves $\tom$ satisfies another nonlinear and nonlocal partial differential equation \cite{cordoba2010interface} (see below for more details).

Our main results can be stated as follows

\begin{theorem}\label{teomuskat}
Let $(\textbf{z},u_\pm,p_\pm)$ be a smooth solution of the two-phase Muskat problem (equations \eqref{eq for dt z}-\eqref{muskat}) such that $\bf{z}$ does not have any self-intersection
and verifying
$$
\sup\limits_{\al,\b\in\BR, t\in[0,T]}\f{|\al-\b|}{|\bz(\al)-\bz(\b)|}<\infty.
$$
Then we have that
\begin{itemize}
\item In the case $\mu_+=\mu_-$, if
$$
\textbf{z}\in C^{2+2\text{sgn}(\gamma)}([0,T]\times \mathbb{R}),
$$ the interface $\textbf{z}$ cannot touch the impervious bottom in $[0,T]$.
\item In the case $\mu_+\neq\mu_-$, if
$$
\textbf{z}\in C^{2+2\text{sgn}(\gamma)}([0,T]\times\BR)
$$
and
$$
(u_+-u_-)\cdot \partial_\alpha\textbf{z}\in C^1([0,T]\times \Gamma(t)),
$$
the interface $\textbf{z}$ cannot touch the impervious bottom in $[0,T]$.
\end{itemize}
\end{theorem}

We do not expect the regularity assumptions in this result to be sharp. For instance, in the case with $\mu_+=\mu_-$, we think that it is enough to have 
$$
\textbf{z}\in C^{1,2+2\text{sgn}(\gamma)}_{t,\alpha},
$$
or even a less strict condition. However, we prefer this statement for the sake of simplicity. 

Furthermore, in addition to the previous result, we are actually able to establish the decay of the interface under certain hypothesis on the physical setting (see Proposition \ref{decayprop}) which seems as a first step towards a finite time singularity result (see also \cite{castro2013breakdown}).

\begin{theorem}\label{teowaves}
Let $(\textbf{z},u_\pm,p_\pm)$ be a smooth solution of the two-phase irrotational Euler problem (equations \eqref{eq for dt z}-\eqref{waterwaves}) such that $\bf{z}$ does not have any self-intersection and verifying
$$
\sup\limits_{\al,\b\in\BR, t\in[0,T]}\f{|\al-\b|}{|\bz(\al)-\bz(\b)|}<\infty.
$$
Then we have that if
$$
\textbf{z}\in C^2([0,T]\times\BR)
$$
and
$$
(u_+-u_-)\cdot \partial_\alpha \textbf{z}\in C^1([0,T]\times \Gamma(t)),
$$
the interface $\textbf{z}$ cannot touch the impervious bottom in $[0,T]$.
\end{theorem}

As before, we do not think that these regularity assumptions are sharp.

These theorems can be understood as \emph{continuation criteria} for the free boundary problems under consideration. Our results generalize the one by Daniel Coutand in \cite{coutand2019finite} in several ways. On the one hand, we give a unitary approach that serves to every free boundary problem in an incompressible and irrotational flow. In particular, it also covers the case of a Muskat flow that, to the best of authors' knowledge, was open. On the other hand, our results improve the criterion on \cite{coutand2019finite} in the sense that the quantities that prevent the vanishing depth scenario are related to the interface and not the bulk of the fluids. In other words, while previously the whole velocity in $\Omega_-$ must satisfy certain smoothness condition, our result only requires the jump of the tangential velocity at the interface to be smooth.

\section{Derivation of the governing equation for the interface}
Let us briefly sketch how to obtain the contour equation formulation of the previous systems. The free boundary is transported by the normal velocity of the fluids. In particular, the evolution of the curve $\G(t)$ is governed by
\beq\label{move of z}
\pa_t\bz\cdot (\pa_\al \bz)^\bot =u_\pm\cdot (\pa_\al \bz)^\bot.
\eeq
Then we introduce the global velocity,
\beqo
v=u_+\chi_{\Om_+}+u_-\chi_{\Om_-}.
\eeqo
The vorticity is defined in the sense of distributions as a scalar function
\beqo
\om=\pa_xv^2-\pa_yv^1.
\eeqo

Note that $\om$ is a measure supported only on $\G(t)$, and we can compute it explicitly. Take any $\phi\in C_c^1(\BR^2_+)$,
\beq
\begin{split}
\int_{\BR^2_+} \om\phi\,dxdy&=\int_{\BR^2_+}(v^1\pa_y\phi-v^2\pa_x\phi)\,dxdy\\
&=\int_{\Om_+} (v^1\pa_y\phi-v^2\pa_x\phi)\,dxdy+\int_{\Om_-}(v^1\pa_y\phi-v^2\pa_x\phi)\,dxdy\\
&=\int_{\BR} \phi (v_--v_+)\cdot (\pa_\al z_1,\pa_\al z_2)\,d\al\\
&=\int_{\BR} \phi\widetilde{\om}\,d\al.
\end{split}
\eeq
with $\widetilde{\om}$ is set as
\beq\label{tom}
\widetilde{\om}=-\jump{v\cdot\pa_\al \bz},
\eeq
where $\jump{F}=F_+-F_-$ denotes the jump of $F$.

Using the Biot-Savart formula for the upper-half plane, we have
\beqo
u(x,y)=\f1{2\pi}\int_{\BR}\left[ \f{(-(y-z_2(s)), x-z_1(s))}{(x-z_1(s))^2+(y-z_2(s))^2}-\f{(-(y+z_2(s)),x-z_1(s))}{(x-z_1(s))^2+(y+z_2(s))^2} \right]\tom(s)\,ds
\eeqo
Then on $\G(t)$, using the Plemelj formula, we have
\beq\label{formula for u}
\begin{split}
u_\pm(\bz(\al))=& \f1{2\pi}p.v.\int_\BR \bigg[ \f{(\bz(\al)-\bz(s))^\bot}{|\bz(\al)-\bz(s)|^2}-\f{(\bz(\al)-(z_1(s),-z_2(s)))^\bot}{|\bz(\al)-(z_1(s),-z_2(s))|^2} \bigg]\tom(s)\,ds \mp\f12 \f{\pa_\al \bz}{|\pa_\al \bz|^2}\tom(\al)
\end{split}
\eeq
Substituting \eqref{formula for u} into \eqref{move of z} gives
\beq\label{eq for dt zv2}
\begin{split}
\pa_t \bz(\al,t)=& \f1{2\pi}p.v.\int_\BR \bigg[ \f{(\bz(\al)-\bz(s))^\bot}{|\bz(\al)-\bz(s)|^2}-\f{(\bz(\al)-(z_1(s),-z_2(s)))^\bot}{|\bz(\al)-(z_1(s),-z_2(s))|^2} \bigg]\tom(s)\,ds+c(\al,t)\pa_\al \bz.
\end{split}
\eeq
Here the quantity $c(\al,t)$ represents the change of parametrization along $\bz$.
In particular,
\beq\label{eq for dt z2}
\pa_t z_2=\f1{2\pi}p.v.\int_\BR \bigg[ \f{z_1(\al)-z_1(s)}{|\bz(\al)-\bz(s)|^2}-\f{z_1(\al)-z_1(s)}{|\bz(\al)-(z_1(s),-z_2(s))|^2} \bigg] \tom(s) \,ds+c(\al,t)\pa_\al z_2.
\eeq
The particular expression for $\tom$ depends on the model under consideration. Namely,

\begin{itemize}
\item in the case of the Muskat problem, $\tom$ satisfies the following nonlinear integral equation \cite{cordoba2011interface}
\begin{multline}\label{muskat}
\frac{\jump{\mu}}{2\pi}\text{p.v.}\int_\BR \bigg[ \f{z_1(\al)-z_1(s)}{|\bz(\al)-\bz(s)|^2}-\f{z_1(\al)-z_1(s)}{|\bz(\al)-(z_1(s),-z_2(s))|^2} \bigg] \tom(s) \,ds\cdot \pa_\al\bz(\alpha)\\
+\left(\frac{\mu_-+\mu_+}{2}\right)\tom(\alpha)=
\gamma\pa_\al\kappa+\jump{\rho}g\pa_\al z_2(\alpha).
\end{multline}
If we assume that the viscosities are equal, the previous equation simplifies to
$$
\mu\tom(\alpha)=
\gamma\pa_\al\kappa+\jump{\rho}g\pa_\al z_2(\alpha).
$$
In this latter case and without lossing generality we will take $\jump{\mu}=0$.
\item in the case of the irrotational Euler equations, following \cite{cordoba2010interface}, we have that $\tom$ satisfies
\begin{align}
\partial_t \tom(\alpha)&=-\partial_\alpha\bigg{[}
 \frac{\frac{\jump{\rho}}{\rho_++\rho_-}}{4\pi^2} \left|\int_\BR \bigg[ \f{(\bz(\al)-\bz(s))^\bot}{|\bz(\al)-\bz(s)|^2}-\f{(\bz(\al)-(z_1(s),-z_2(s)))^\bot}{|\bz(\al)-(z_1(s),-z_2(s))|^2} \bigg]\tom(s)\,ds\right|^2
 -\frac{\frac{\jump{\rho}}{\rho_++\rho_-}}{4} \frac{\tom^2}{|{\partial_\alpha} \bz|^2}
\nonumber\\
&\quad \qquad
+  \frac{\frac{\jump{\rho}}{\rho_++\rho_-}}{\pi}  \int_\BR \bigg[ \f{(\bz(\al)-\bz(s))^\bot}{|\bz(\al)-\bz(s)|^2}-\f{(\bz(\al)-(z_1(s),-z_2(s)))^\bot}{|\bz(\al)-(z_1(s),-z_2(s))|^2} \bigg]\tom(s)\,ds  \cdot \partial_ \alpha \bz(\alpha) c( \alpha)
\nonumber \\
&\qquad \qquad
-c( \alpha) \tom( \alpha )  - \frac{2\gamma\kappa}{\rho^++\rho^-}  -2 \frac{\jump{\rho}}{\rho_++\rho_-} g z_2 \bigg{]} \nonumber \\
& \qquad \qquad
+\frac{\frac{\jump{\rho}}{\rho_++\rho_-}}{\pi} \partial_t\bigg{[}
 \int_\BR \bigg[ \f{(\bz(\al)-\bz(s))^\bot}{|\bz(\al)-\bz(s)|^2}-\f{(\bz(\al)-(z_1(s),-z_2(s)))^\bot}{|\bz(\al)-(z_1(s),-z_2(s))|^2} \bigg]\tom(s)\,ds  \cdot \partial_ \alpha \bz(\alpha)
\bigg{]} \,.
\label{waterwaves}
\end{align}

\end{itemize}

\section{The cornerstone}
The results in this paper are mainly obtained as a consequence of the following cornerstone theorem:
\begin{theorem}\label{teo1}
Let $(\bz, \tom)$ be a couple satisfying \eqref{eq for dt z} on some time interval $[0,T]$. Assume also that there is a constant $A>0$ such that
\begin{align}
\label{regularity condition}&\|\bz(\al,t)-(\al,1)\|_{C^2(\BR\times [0,T])}\leq A,\\
\label{regularity condition2}&\|\tom\|_{C^1(\BR\times [0,T])}\leq A,\\
\label{behavior at infty}& \lim\limits_{|\al|\ri\infty} |\bz(\al,t)-(\al,1)|=0, \;\;\forall t\in[0,T],\\
\label{chord-arc condition}& \sup\limits_{\al,\b\in\BR, t\in[0,T]}\f{|\al-\b|}{|\bz(\al)-\bz(\b)|}\leq A.
\end{align}

 Then there is a constant $C(A)$ such that
\beqo
\min\limits_{\al\in\BR} |z_2(\al,t)|\geq e^{-Ce^{Ct}}, \quad \forall t\in [0,T],
\eeqo
which means $\bz(\al,t)$ cannot touch the bottom in finite time.
\end{theorem}
\begin{proof}[Proof of Theorem \ref{teo1}]
The proof has the same flavour as the one in \cite{gancedo2014absence}. The argument works by contradiction. Let us assume that the solution satisfies
$$
\bz\in C^2([0,T]\times\mathbb{R})
$$
and define
\beq\label{def of h}
m(t)=z_2(\al_t,t)=\min\limits_{\al\in \BR} z_2(\al,t).
\eeq
Using the regularity of $\bf{z}$, we have that $m(t)$ is a Lipschitz function and, as a consequence, it is almost everywhere differentiable with derivative given by
$$
\frac{d}{dt}m(t)=\partial_t z_2(\alpha_t,t).
$$
From \eqref{eq for dt z2} and $\pa_\al z_2(\al_t,t)=0$ we obtain that
\beq\label{eq of h}
\begin{split}
\f{d}{dt} m(t)=\f{1}{2\pi}p.v.\int_\BR \bigg[ \f{z_1(\al_t)-z_1(s)}{|\bz(\al_t)-\bz(s)|^2}-\f{z_1(\al_t)-z_1(s)}{|\bz(\al_t)-(z_1(s),-z_2(s))|^2} \bigg] \tom(s) \,ds
\end{split}
\eeq
We define
\begin{align*}
J&:= p.v.\int_\BR \bigg[ \f{z_1(\al)-z_1(s)}{|\bz(\al)-\bz(s)|^2}-\f{z_1(\al)-z_1(s)}{|\bz(\al)-(z_1(s),-z_2(s))|^2} \bigg] \tom (s)\,ds
\end{align*}
The goal is to show that
\beqo
|J|\leq C m\log\left(\f{1}{m}\right).
\eeqo

We decompose it as follows:
$$
J=J_m+J_1+J_\infty,
$$
with
\begin{align*}
J_m&=  p.v.\int_{|s|<m}\bigg[ \f{\tom(\al_t-s)(z_1(\al_t)-z_1(\al_t-s))}{|\bz(\al_t)-\bz(\al_t-s)|^2}-\f{\tom(\al_t-s)(z_1(\al_t)-z_1(\al_t-s))}{|\bz(\al_t)-(z_1(\al_t-s),-z_2(\al_t-s))|^2}\bigg]\,ds \\
J_1&=  \int_{m\leq |s|<1} \bigg[ \f{\tom(\al_t-s)(z_1(\al_t)-z_1(\al_t-s))}{|\bz(\al_t)-\bz(\al_t-s)|^2}-\f{\tom(\al_t-s)(z_1(\al_t)-z_1(\al_t-s))}{|\bz(\al_t)-(z_1(\al_t-s),-z_2(\al_t-s))|^2}\bigg]\,ds \\
J_\infty&= p.v.\int_{|s|\geq 1} \bigg[ \f{\tom(\al_t-s)(z_1(\al_t)-z_1(\al_t-s))}{|\bz(\al_t)-\bz(\al_t-s)|^2}-\f{\tom(\al_t-s)(z_1(\al_t)-z_1(\al_t-s))}{|\bz(\al_t)-(z_1(\al_t-s),-z_2(\al_t-s))|^2}\bigg]\,ds \bigg).
\end{align*}

For any function $g$ and points $a,b\in \BR$, we introduce the following abbreviations
\beqo
\d_g(a,b):=g(a)-g(b),\qquad \s_g(a,b):=g(a)+g(b).
\eeqo
Here $g$ is either a scalar function or a vector function.

For the term $J_m$, we first write
$$
J_m=J_m^1+J_m^2
$$
with
\begin{align*}
J_m^1= &  p.v.\int_{0<|s|<m}  \f{\tom(\al_t-s)(z_1(\al_t)-z_1(\al_t-s))}{|\bz(\al_t)-\bz(\al_t-s)|^2}\,ds\\
J_m^2=&- p.v. \int_{0<|s|<m} \f{\tom(\al_t-s)(z_1(\al_t)-z_1(\al_t-s))}{|\bz(\al_t)-(z_1(\al_t-s),-z_2(\al_t-s))|^2}\,ds.
\end{align*}
Then we estimate
\beq\label{estimate of Jh1}
\begin{split}
J_m^1=&\lim\limits_{\e\ri 0}\int_\e^m \left[ \f{\tom(\al_t-s)(z_1(\al_t)-z_1(\al_t-s))}{|\bz(\al_t)-\bz(\al_t-s)|^2}+\f{\tom(\al_t+s)(z_1(\al_t)-z_1(\al_t+s))}{|\bz(\al_t)-\bz(\al_t+s)|^2}  \right]\,ds\\
=&\lim\limits_{\e\ri 0}\int_\e^m  \bigg[\f{|\d_{\bz}(\al_t,\al+s)|^2(\d_{z_1}(\al_t,\al_t-s))\tom(\al_t-s)}{|\d_{\bz}(\al_t,\al_t-s)|^2|\d_{\bz}(\al_t,\al_t+s)|^2}\\
&\qquad\qquad\quad +\f{|\d_{\bz}(\al_t,\al_t-s)|^2(\d_{z_1}(\al_t,\al_t+s))\tom(\al_t+s)}{|\d_{\bz}(\al_t,\al_t-s)|^2|\d_{\bz}(\al_t,\al_t+s)|^2}\bigg]\,ds \\
=& \lim\limits_{\e\ri 0}\int_\e^m \bigg[ \f{\d_{\tom}(\al_t-s,\al_t+s)|\d_{\bz}(\al_t,\al_t+s)|^2\d_{z_1}(\al_t,\al_t-s)}{|\d_{\bz}(\al_t,\al_t-s)|^2|\d_{\bz}(\al_t,\al_t+s)|^2} \\
& \qquad\qquad +   \f{\tom(\al_t+s)\left( |\d_{\bz}(\al_t,\al_t+s)|^2\d_{z_1}(\al_t,\al_t-s)+|\d_{\bz}(\al_t,\al_t-s)|^2\d_{z_1}(\al_t,\al_t+s) \right)}{|\d_{\bz}(\al_t,\al_t-s)|^2|\d_{\bz}(\al_t,\al_t+s)|^2}  \bigg]\,ds.
\end{split}
\eeq
Due to \eqref{regularity condition}, \eqref{chord-arc condition} together with the regularity of $\tom$ and the fact that
$$
\partial_\alpha z_2(\al_t,t)=0,
$$
we have
\begin{align*}
|\d_{\tom}(\al_t-s,\al_t+s)|&\leq 2\|\tom\|_{C^1}s,\\
|\d_{\bz}(\al_t,\al_t+ s)|^2 &\geq \f{1}{A^2} s^2,\\
|\d_{\bz}(\al_t,\al_t- s)|^2 &\geq \f{1}{A^2} s^2,\\
|\d_{z_1}(\al_t,\al_t+ s)|&\leq \|z_1\|_{C^1} s,\\
|\d_{z_1}(\al_t,\al_t- s)|&\leq \|z_1\|_{C^1} s,\\
|\d_{z_2}(\al_t,\al_t+ s)|&\leq \f12\|z_2\|_{C^2}s^2\\
|\d_{z_2}(\al_t,\al_t- s)|&\leq \f12\|z_2\|_{C^2}s^2\\
|\d_{z_1}(\al_t,\al_t+s)+\d_{z_1}(\al_t,\al_t-s)|&\leq \|z_1\|_{C^2}s^2.
\end{align*}
Furthermore, we compute
\begin{multline*}
|\d_{\bz}(\al,\al+s)|^2\d_{z_1}(\al,\al-s)+|\d_{\bz}(\al,\al-s)|^2\d_{z_1}(\al,\al+s) \\
= \d_{z_1}(\al,\al+s)\d_{z_1}(\al,\al-s)(\d_{z_1}(\al,\al+s)+\d_{z_1}(\al,\al-s))\\
 +\d_{z_2}(\al,\al+s)^2\d_{z_1}(\al,\al-s)+\d_{z_2}(\al,\al-s)^2\d_{z_1}(\al,\al+s).
\end{multline*}
Substituting all these estimates into \eqref{estimate of Jh1}, it follows that
\beq\label{estimate of Jh1 continue}
\begin{split}
|J_m^1|&\leq C\int_0^m \left[ \f{(\|\tom\|_{C^1}\|z_1\|_{C^1})s^2}{\f{1}{A^2}s^2}+\f{\|\tom\|_{C^0}(\|z_1\|_{C_2}^3s^4+\|z_2\|_{C^2}^2\|z_1\|_{C_1}s^5)}{\f{1}{A^4}s^4}\right]\,ds\\
&\leq C(A, \|\tom\|_{C^1}, \|\bz\|_{C^2}) m.
\end{split}
\eeq

Similarly, for $J_m^2$ we compute
\beq\label{estimate of Jh2}
\begin{split}
J_m^2=&\lim\limits_{\e\ri 0}\int_\e^m \left[ \f{\tom(\al_t-s)\d_{z_1}(\al_t,\al_t-s)}{\d_{z_1}(\al_t,\al_t-s)^2+\s_{z_2}(\al_t,\al_t-s)^2}+\f{\tom(\al_t+s)\d_{z_1}(\al_t,\al_t+s)}{\d_{z_1}(\al_t,\al_t+s)^2+\s_{z_2}(\al_t,\al_t+s)^2}  \right]\,ds\\
=& \lim\limits_{\e\ri 0}\int_\e^m \bigg[ \f{\d_{\tom}(\al_t-s,\al_t+s)(\d_{z_1}(\al_t,\al_t+s)^2+\s_{z_2}(\al_t,\al_t+s)^2)\d_{z_1}(\al_t,\al_t-s)}{(\d_{z_1}(\al_t,\al_t+s)^2+\s_{z_2}(\al_t,\al_t+s)^2)(\d_{z_1}(\al_t,\al_t-s)^2+\s_{z_2}(\al_t,\al_t-s)^2)} \\
&  \qquad\qquad +   \f{\tom(\al_t+s) \d_{z_1}(\al_t,\al_t-s)(\d_{z_1}(\al_t,\al_t+s)^2+\s_{z_2}(\al_t,\al_t+s)^2)}{(\d_{z_1}(\al_t,\al_t+s)^2+\s_{z_2}(\al_t,\al_t+s)^2)(\d_{z_1}(\al_t,\al_t-s)^2+\s_{z_2}(\al_t,\al_t-s)^2)}\\
&\qquad\qquad +\f{\tom(\al_t+s)\d_{z_1}(\al_t,\al_t+s)(\d_{z_1}(\al_t,\al_t-s)^2+\s_{z_2}(\al_t,\al_t-s)^2) }{(\d_{z_1}(\al_t,\al_t+s)^2+\s_{z_2}(\al_t,\al_t+s)^2)(\d_{z_1}(\al_t,\al_t-s)^2+\s_{z_2}(\al_t,\al_t-s)^2)}  \bigg]\,ds.
\end{split}
\eeq

We can write
\begin{multline}\label{Jh2 2nd term}
\d_{z_1}(\al_t,\al_t-s)(\d_{z_1}(\al_t,\al_t+s)^2\\
+\s_{z_2}(\al_t,\al_t+s)^2)+\d_{z_1}(\al_t,\al_t+s)(\d_{z_1}(\al_t,\al_t-s)^2+\s_{z_2}(\al_t,\al_t-s)^2)\\
= (\d_{z_1}(\al_t,\al_t-s)+\d_{z_1}(\al_t,\al_t+s))\d_{z_1}(\al_t,\al_t-s)\d_{z_1}(\al_t,\al_t+s)\\
 + \d_{z_1}(\al_t,\al_t-s)(\s_{z_2}(\al_t,\al_t+s)^2-(2z_2(\al_t))^2)\\
 +\d_{z_1}(\al_t,\al_t+s)(\s_{z_2}(\al_t,\al_t-s)^2-(2z_2(\al_t))^2)\\
 + (\d_{z_1}(\al_t,\al_t-s)+\d_{z_1}(\al_t,\al_t+s))(2z_2(\al_t))^2.
\end{multline}
We also need the following estimates, valid for $0<s<m$,
\begin{align}
\label{Jh2 a} & \qquad\qquad |\s_{z_2}(\al_t,\al_t\pm s)|\geq 2m,\\
\label{Jh2 b} & |\s_{z_2}(\al_t,\al_t\pm s)^2-(2z_2(\al_t))^2|\leq \|z_2\|_{C^2} s^2(3m+ \f12\|z_2\|_{C^2}s^2)\leq \|z_2\|_{C^2} 4s^2m,
\end{align}
when $m$ is sufficiently small. Therefore, substituting \eqref{Jh2 2nd term}, \eqref{Jh2 a} and \eqref{Jh2 b} into \eqref{estimate of Jh2}, we have
\beq\label{estimate of Jh2 continue}
\begin{split}
|J_m^2| \leq &\left|\int_0^m \f{\|\tom\|_{C^1}\|z_1\|_{C^1}s^2}{4m^2}\,ds\right| \\
&\quad +\left| \int_0^m \f{ \|\tom\|_{C^0}(\|z_1\|_{C^2}^3s^4+ 8\|z_1\|_{C^1}\|z_2\|_{C^2}ms^3)+4\|z_1\|_{C^2} s^2 m^2  }{16m^4}\,ds \right|\\
\leq &C(\|\tom\|_{C^1}, \|\bz\|_{C^2}) m.
\end{split}
\eeq
Collecting \eqref{estimate of Jh1 continue} and \eqref{estimate of Jh2 continue} we conclude that
\beq\label{estimate of Jh}
|J_m|\leq C(A, \|\tom\|_{C^1}, \|\bz\|_{C^2})m.
\eeq

For the term $J_1$, for the sake of convenience, we define the following quantity
\beq\label{def of P}
P(\al,s):=|\bz(\al)-\bz(s)|^2|\bz(\al)-(z_1(s),-z_2(s))|^2
\eeq
Then we write
\beq\label{compute J1}
\begin{split}
J_1=& \int_{m<|\al_t-s|<1} \left[ \f{1}{|\bz(\al_t)-\bz(s)|^2}-\f{1}{(z_1(\al_t)-z_1(s))^2+(z_2(\al_t)+z_2(s))^2} \right]\tom(s)\d_{z_1}(\al_t,s)\,ds\\
=& \int_{m\leq |s|\leq 1} \f{4mz_2(\al_t-s)\d_{z_1}(\al_t,\al_t-s)}{P(\al_t,\al_t-s)}\tom(\al_t-s)    \,ds\\
=&4m\int_m^1 \f{z_2(\al_t-s)\d_{z_1}(\al_t,\al_t-s)}{P(\al_t,\al_t-s)}\d_{\tom}(\al_t-s,\al_t+s)\,ds\\
&+ 4m\int_m^1 \tom(\al_t+s)\left[ \f{z_2(\al_t-s)\d_{z_1}(\al_t,\al_t-s)}{P(\al_t,\al_t-s)}+\f{z_2(\al_t+s)\d_{z_1}(\al_t,\al_t+s)}{P(\al_t,\al_t+s)} \right]\,ds\\
=&J_1^1+J_1^2.
\end{split}
\eeq
We have the following simple estimates
\begin{align}\label{est z_2}
&z_2(\al-s)\leq m+\|z_2\|_{C^1}s\leq s(1+\|z_2\|_{C^1}),\\
\label{est P} &\f{1}{A^4} s^4\leq  P(\al,\al\pm s)\leq C(\|\bz\|_{C^1}) s^4,
\end{align}
when $m\leq s\leq 1$. It follows that
\beq\label{estimate of J11}
|J_1^1|\leq 4m(1+\|z_2\|_{C^1})\|z_1\|_{C_1}\|\tom\|_{C^1}A^4 \int_m^1\f{s^3}{s^4}\,ds\leq C(A, \|\bz\|_{C^1}, \|\tom\|_{C^1}) m\log{\f{1}{m}}.
\eeq

For $J_1^2$, as
$$
P(\al_t,\al_t\pm s)\geq \f{1}{A^4}s^4,
$$
it suffices to estimate
$$z_2(\al_t-s)\d_{z_1}(\al_t,\al_t-s)P(\al_t,\al_t+s)+ z_2(\al_t+s)\d_{z_1}(\al_t,\al_t+s)P(\al_t,\al_t-s).
$$
We have that
\begin{multline}\label{estimate of J12}
|z_2(\al_t-s)\d_{z_1}(\al_t,\al_t-s)P(\al_t,\al_t+s)+z_2(\al_t+s)\d_{z_1}(\al_t,\al_t+s)P(\al_t,\al_t-s)|\\
\leq |(\d_{z_1}(\al_t,\al_t-s)+\d_{z_1}(\al_t,\al_t+s))z_2(\al_t-s)P(\al_t,\al_t+s)|\\
 + |\d_{z_1}(\al_t,\al_t+s)[z_2(\al_t+s)P(\al_t,\al_t-s)-z_2(\al_t-s)P(\al_t,\al_t+s)]|\\
\leq C(\|\bz\|_{C^2})s^7+ |\d_{z_1}(\al_t,\al_t+s)\d_{z_2}(\al_t+s,\al_t-s)P(\al_t,\al_t-s)|\\
+|\d_{z_1}(\al_t,\al_t+s)z_2(\al_t-s)(P(\al_t,\al_t-s)-P(\al_t,\al_t+s))|\\
\leq C(\|\bz\|_{C^2})s^7+ C(\|\bz\|_{C^2})s^2 |P(\al,\al-s)-P(\al,\al+s)|,
\end{multline}
where we have used $m<s$. Now the only problem left becomes the estimation of
$$|P(\al_t,\al_t-s)-P(\al_t,\al_t+s)|.
$$ By definition

\begin{equation*}
\begin{split}
|P(\al_t,\al_t-s)-P(\al_t,\al_t+s)|=& \bigg| (\d_{z_1}(\al_t,\al_t-s)^2+\d_{z_2}(\al_t,\al_t-s)^2) (\d_{z_1}(\al_t,\al_t-s)^2 +\s_{z_2}(\al_t,\al_t-s)^2)\\
&-(\d_{z_1}(\al_t,\al_t+s)^2+\d_{z_2}(\al_t,\al_t+s)^2) (\d_{z_1}(\al_t,\al_t+s)^2 +\s_{z_2}(\al_t,\al_t+s)^2)\bigg|\\
=& P_1+P_2+P_3+P_4.
\end{split}
\end{equation*}
with
\begin{align*}
P_1=& | \d_{z_1}(\al_t,\al_t-s)^4-\d_{z_1}(\al_t,\al_t+s)^4|\\
P_2=&|\d_{z_1}(\al_t,\al_t-s)^2\d_{z_2}(\al_t,\al_t-s)^2-\d_{z_1}(\al_t,\al_t+s)^2\d_{z_2}(\al_t,\al_t+s)^2|\\
P_3=& |\d_{z_1}(\al_t,\al_t-s)^2\s_{z_2}(\al_t,\al_t-s)^2-\d_{z_1}(\al_t,\al_t+s)^2\s_{z_2}(\al_t,\al_t+s)^2|\\
P_4=&|\d_{z_2}(\al_t,\al_t-s)^2\s_{z_2}(\al_t,\al_t-s)^2-\d_{z_2}(\al_t,\al_t+s)^2\s_{z_2}(\al_t,\al_t+s)^2|.
\end{align*}

For the term $P_2$ and $P_4$, since they contain the term $\d_{z_2}(\al,\al\pm s)^2$ which is of order $s^4$, we get that
\beqo
P_2+P_4\leq C(\|z_2\|_{C^2}, \|z_1\|_{C^1}) s^6.
\eeqo
For $P_1$ we have that
\begin{align*}
P_1& =|\d_{z_1}(\al,\al-s)^2+\d_{z_1}(\al,\al+s)^2|\cdot|\d_{z_1}(\al,\al-s)^2-\d_{z_1}(\al,\al+s)^2| \\
&\leq 2\|z_1\|_{C^1}^2 s^2 |\d_{z_1}(\al,\al-s)+\d_{z_1}(\al,\al+s)|\d_{z_1}(\al-s,\al+s)|\\
&\leq C(\|z_1\|_{C^2}) s^5.
\end{align*}
Similarly, for $P_3$ we estimate that
\begin{align*}
P_3 &=|(\d_{z_1}(\al,\al-s)^2-\d_{z_1}(\al,\al+s)^2)\s_{z_2}(\al,\al-s)^2|\\
&\quad + |\d_{z_1}(\al,\al+s)^2(\s_{z_2}(\al,\al-s)^2-\s_{z_2}(\al,\al+s)^2)|\\
&\leq C(\|z_1\|_{C^2},\|z_2\|_{C^1})s^5+ C(\|z_2\|_{C^2},\|z_1\|_{C^1}) s^5.
\end{align*}

Substituting all the estimates of $P_i$ into \eqref{estimate of J12} and using \eqref{compute J1}, we conclude that
\beq\label{estimate of J12 continue}
|J_1^2|\leq 4m \|\tom\|_{C^0} C(\|\bz\|_{C^2})A^8 \int_m^1 \f{s^7}{s^8}\,ds\leq C(A, \|\bz\|_{C^2}, \|\tom \|_{C^0}) m\log{\f1m}.
\eeq
Estimate \eqref{estimate of J12 continue} together with \eqref{estimate of J11} leads to
\beq\label{estimate of J1}
|J_1|\leq C(A,\|\bz\|_{C^2},\|\tom\|_{C^1})m\log{\f1m}.
\eeq

Finally, for the term $J_\infty$, we compute
\beq\label{estimate of Jinfty}
\begin{split}
|J_\infty|&=\left|p.v.\int_{|\al_t-s|>1}  \f{4\d_{z_1}(\al_t,s)\tom(s)z_2(\al_t)z_2(s)}{P(\al_t,s)} \,ds\right|\\
&\leq 4m\|z_1\|_{C^1}\|\tom\|_{C^0}\|z_2\|_{L^\infty}A^4 \int_{|\al_t-s|>1} \f{|\al_t-s|}{|\al_t-s|^4}\,ds\\
&\leq \left( 2 A^4 \|z_1\|_{C^1}\|\tom\|_{C^0}\|z_2\|_{L^\infty} \right)m.
\end{split}
\eeq

Finally, \eqref{estimate of Jh}, \eqref{estimate of J1} and \eqref{estimate of Jinfty} together imply that
\beqo
\left|\f{d}{dt} m\right|\leq C(\|\bz\|_{C^2},\|\tom\|_{C^1},A)m\log{\f1m},
\eeqo
where the constant $C$ only depends on $\|\bz\|_{C^2}$, $\|\tom\|_{C^1}$ and the constant $A$ in the chord-arc condition \eqref{chord-arc condition}. It follows that
\beqo
m(t)\geq e^{-Ce^{Ct}},
\eeqo
which completes the proof of the theorem.
\end{proof}

\section{The Muskat problem}
We start this section proving that the solution to the Muskat problem in the Rayleigh-Taylor unstable regime where the heavier fluid is on top of the lighter fluid can actually approach the impervious bottom. Namely, we have the following result
\begin{proposition}\label{decayprop}
Let $(z,\tom)$ be an analytical solution of the Muskat problem \eqref{eq for dt z}-\eqref{muskat} with
$$
\g=0, \jump{\mu}=0
$$ in the Rayleigh-Taylor unstable case
$$
\jump{\rho}>0.
$$
Assume that
$$
\pa_sz_1(s)\geq 0.
$$
Then
\beq\label{h(t) decrease}
\f{d}{dt}m(t)\leq 0,
\eeq
where
$$
m(t)=\min_{s\in\BR} z_2(s,t).
$$
\end{proposition}
\begin{proof}
We observe that  even if in the RT unstable the Muskat problem is ill-posed in Sobolev spaces in absence of capillary forces, the solution exists and is unique in the analytic setting \cite{cordoba2014confined}. We define $\al_t$ as in \eqref{def of h}. Due to the regularity of $z$ we have that $m$ satisfies the equation \eqref{eq of h}. By \eqref{eq of h} and $\g=0$ it suffices to show that, after taking
$$
\jump{\rho}g=2\pi
$$
without loss of generality,
\beq\label{ineq for I}
0\geq I:=p.v.\int_\BR \f{\pa_s z_2(s)(z_1(\al_t)-z_1(s))}{|\bz(\al_t)-\bz(s)|^2}-\f{\pa_s z_2(s)(z_1(\al_t)-z_1(s))}{|\bz(\al_t)-(z_1(s),-z_2(s))|^2}\,ds.
\eeq
We define
\beq\label{def of tilde I}
\tilde{I}:= p.v.\int_\BR \f{\pa_s z_1(s)(z_2(\al_t)-z_2(s))}{|\bz(\al_t)-\bz(s)|^2}+\f{\pa_s z_1(s)(z_2(\al_t)+z_2(s))}{|\bz(\al_t)-(z_1(s),-z_2(s))|^2}\,ds.
\eeq
We first show that
\beq\label{relation of I and tildeI}
\tilde{I}-I=\pi
\eeq

To prove this, we use the basic tools from complex integrals. Write
\beqo
z(s):=z_1(s)+iz_2(s).
\eeqo
$\G(t)$ is a curve in the complex plane. Direct calculation implies that
\beq\label{complex integral}
\tilde{I}-I=\lim\limits_{r\ri 0,R\ri\infty}\mathrm{Im}\left( \int_{\G\cap \{r\leq |z-z(\al)|\leq R}  \f{dz}{z-z(\al_t)}+\f{dz}{\bar{z}(\al_t)-z} \right),
\eeq
where $\mathrm{Im}$ means the imaginary part of a complex number. We define
\beqo
\G_{r,R}:=\left(\G(t)\cap \{r\leq |z-z(\al_t)|\leq R\})\right)\cap \left( \{|z-z(\al_t)|=r\}\cap \Om_+(t) \right)\cap \left(\{|z-z(\al_t)|=R\}\cap \Om_+(t) \right),
\eeqo
for $ r\ll 1,\; R\gg 1$.  According to \eqref{behavior at infty} and the regularity assumption of $\G(t)$, we know that $\G(t)$ is very flat near the minimal point $z(\al)$ and at infinity, which further implies that $\G_{r,R}$ is a well-defined simple closed curve when $r$ is suitably small and $R$ is appropriately large. It is obvious that the functions
$$
\f{1}{z-z(\al_t)}
$$
and
$$\f{1}{\bar{z}(\al_t)-z}
$$
do not contain any poles in the region enclosed by $\G_{r,R}$. We can use Cauchy's integral theorem to derive that
\beq\label{cauchy theorm 1}
\begin{split}
0=&\lim\limits_{r\ri 0,R\ri\infty} \int_{\G_{r,R}} \f{dz}{z-z(\al_t)}\\
=&\lim\limits_{r\ri 0,R\ri\infty} \int_{\G\cap \{r\leq |z-z(\al_t)|\leq R\}} \f{dz}{z-z(\al_t)}+\lim\limits_{R\ri\infty}\int_{\{|z-z(\al_t)|= R\}\cap\Om_+}\f{dz}{z-z(\al_t)}\\
&\qquad +\lim\limits_{r\ri 0}\int_{\{|z-z(\al_t)|= r\}\cap\Om_+} \f{dz}{z-z(\al_t)}\\
=& \lim\limits_{r\ri 0,R\ri\infty} \int_{\G\cap \{r\leq |z-z(\al_t)|\leq R\}} \f{dz}{z-z(\al_t)}+\pi i-\pi i\\
=& \lim\limits_{r\ri 0,R\ri\infty} \int_{\G\cap \{r\leq |z-z(\al_t)|\leq R\}} \f{dz}{z-z(\al_t)}.
\end{split}
\eeq
Similarly,,
\beq\label{cauchy theorm 2}
\begin{split}
0=&\lim\limits_{r\ri 0,R\ri\infty} \int_{\G_{r,R}} \f{dz}{\bar{z}(\al_t)-z}\\
=&\lim\limits_{r\ri 0,R\ri\infty} \int_{\G\cap \{r\leq |z-z(\al_t)|\leq R\}} \f{dz}{\bar{z}(\al_t)-z}+\lim\limits_{R\ri\infty}\int_{\{|z-z(\al_t)|= R\}\cap\Om_+}\f{dz}{\bar{z}(\al_t)-z}\\
&\qquad +\lim\limits_{r\ri 0}\int_{\{|z-z(\al_t)|= r\}\cap\Om_+} \f{dz}{\bar{z}(\al_t)-z}\\
=& \lim\limits_{r\ri 0,R\ri\infty} \int_{\G\cap \{r\leq |z-z(\al_t)|\leq R\}} \f{dz}{\bar{z}(\al_t)-z}-\pi i.
\end{split}
\eeq
Equation \eqref{relation of I and tildeI} then follows immediately from \eqref{complex integral}, \eqref{cauchy theorm 1} and \eqref{cauchy theorm 2}.

Now we only need to prove $\tilde{I}\leq \pi$. We compute
\beqo
\begin{split}
&\pi-\tilde{I}\\
=& \int_\BR \f{2\pa_s z_1(s)z_2(\al_t)}{|z_1(\al_t)-z_1(s)|^2+(2z_2(\al_t))^2}\,ds \\
 &\quad - p.v.\int_\BR \f{\pa_s z_1(s)(z_2(\al_t)-z_2(s))}{|\bz(\al_t)-\bz(s)|^2}+\f{\pa_s z_1(s)(z_2(\al_t)+z_2(s))}{|\bz(\al_t)-(z_1(s),-z_2(s))|^2}\,ds\\
=&\left( \int_\BR \f{2\pa_s z_1(s)z_2(\al_t)}{|z_1(\al_t)-z_1(s)|^2+(2z_2(\al_t))^2}\,ds -\int_{\BR}\f{2\pa_s z_1(s)z_2(\al)}{|\bz(\al_t)-(z_1(s),-z_2(s))|^2}\,ds\right)\\
&\quad +\bigg( \int_{\BR}\f{2\pa_s z_1(s)z_2(\al_t)}{|\bz(\al_t)-(z_1(s),-z_2(s))|^2}\,ds\\
&\quad\quad - p.v.\int_\BR \f{\pa_s z_1(s)(z_2(\al_t)-z_2(s))}{|\bz(\al_t)-\bz(s)|^2}+\f{\pa_s z_1(s)(z_2(\al_t)+z_2(s))}{|\bz(\al_t)-(z_1(s),-z_2(s))|^2}\,ds\bigg)\\
=& \int_{\BR} \f{2\pa_sz_1(s)z_2(\al_t)(3z_2(\al_t)+z_2(s))(z_2(s)-z_2(\al_t))}{|\bz(\al_t)-(z_1(s),-z_2(s))|^2(|z_1(\al_t)-z_1(s)|^2+(2z_2(\al_t))^2)}\,ds\\
&\quad + \int_\BR \f{4\pa_sz_1(s)z_2(\al_t)z_2(s)(z_2(s)-z_2(\al_t))}{|\bz(\al_t)-\bz(s)|^2|\bz(\al_t)-(z_1(s),-z_2(s))|^2}\,ds.
\end{split}
\eeqo
Using the hypothesis $\pa_sz_1(s)\geq 0$, then the fact that $\pi-\tilde{I}\geq 0$ follows immediately since every term in the final two integrands is non-negative.
\end{proof}

We observe that this proposition generalizes the result in \cite{cordoba2014confined} to the case where the curve has a vertical tangent.

\begin{proof}[Proof of Theorem \ref{teomuskat}]
The case with
$$
\jump{\mu}=0
$$
follows from Theorem \ref{teo1} observing that
$$
\tom(\alpha)=
\gamma\pa_\al\kappa+\jump{\rho}g\pa_\al z_2(\alpha),
$$
so
$$
\|\tom\|_{C^1}\leq C\|z\|_{C^4}
$$
if $\gamma>0$ and
$$
\|\tom\|_{C^1}\leq C\|z\|_{C^2}
$$
if $\gamma=0$.
The case with
$$
\jump{\mu}\neq0
$$
is an application of Theorem \ref{teo1} once we recall that
$$
\|\widetilde{\om}\|_{C^1([0,T]\times \Gamma(t))}=\|\jump{v\cdot\pa_\al \bz}\|_{C^1([0,T]\times \mathbb{R})}\leq A,
$$
by the hypotheses of the theorem.
\end{proof}

\section{The internal waves problem}
\begin{proof}[Proof of Theorem \ref{teowaves}]
This theorem follows from an application of Theorem \ref{teo1} noticing that
$$
\|\widetilde{\om}\|_{C^1([0,T]\times \Gamma(t))}=\|\jump{v\cdot\pa_\al \bz}\|_{C^1([0,T]\times \mathbb{R})}\leq A,
$$
by the hypotheses of the theorem.
\end{proof}

\section*{Acknowledgments} Z.G. was supported by the Basque Government through the BERC 2022-2025 program and by the Spanish State Research Agency through BCAM Severo Ochoa excellence accreditation SEV-2017-0718 and through project PID2020-114189RB-I00 funded by Agencia Estatal de Investigaci\'{o}n (PID2020-114189RB-I00 / AEI / 10.13039/501100011033). R.G-B was supported by the project "Mathematical Analysis of Fluids and Applications" Grant PID2019-109348GA-I00 funded by MCIN/AEI/ 10.13039/501100011033 and acronym "MAFyA". This publication is part of the project PID2019-109348GA-I00 funded by MCIN/ AEI /10.13039/501100011033. R.G-B is also supported by a 2021 Leonardo Grant for Researchers and Cultural Creators, BBVA Foundation. The BBVA Foundation accepts no responsibility for the opinions, statements, and contents included in the project and/or the results thereof, which are entirely the responsibility of the authors. Part of this research was performed when R.G-B was Visiting Fellow of the Basque Center for Applied Mathematics. R.G-B is grateful to Basque Center for Applied Mathematics for their hospitality during this visit. 

\bibliographystyle{abbrv}

\begin{thebibliography}{10}

\bibitem{alvarez2009pinch}
E.~Alvarez-Lacalle, J.~Casademunt, and J.~Eggers.
\newblock Pinch-off singularities in rotating Hele-Shaw flows at high viscosity
  contrast.
\newblock {\em Physical Review E}, 80(5):056306, 2009.

\bibitem{ambrose2021numerical}
D.~M. Ambrose, R.~Camassa, J.~L. Marzuola, R.~M. McLaughlin, Q.~Robinson, and
  J.~Wilkening.
\newblock Numerical algorithms for water waves with background flow over
  obstacles and topography.
\newblock {\em arXiv preprint arXiv:2108.01786}, 2021.

\bibitem{aurther2019rigorous}
C.~Aurther, R.~Granero-Belinch{\'o}n, S.~Shkoller, and J.~Wilkening.
\newblock Rigorous asymptotic models of water waves.
\newblock {\em Water Waves}, 1(1):71--130, 2019.

\bibitem{bae2021singularity}
H.~Bae and R.~Granero-Belinch{\'o}n.
\newblock Singularity formation for the Serre-Green-Naghdi equations and
  applications to abcd-Boussinesq systems.
\newblock {\em Monatshefte f{\"u}r Mathematik}, pages 1--14, 2021.

\bibitem{camassa2019singularity}
R.~Camassa, G.~Falqui, G.~Ortenzi, M.~Pedroni, and G.~Pitton.
\newblock Singularity formation as a wetting mechanism in a dispersionless
  water wave model.
\newblock {\em Nonlinearity}, 32(10):4079, 2019.

\bibitem{camassa2019hydrodynamic}
R.~Camassa, G.~Falqui, G.~Ortenzi, M.~Pedroni, and C.~Thomson.
\newblock Hydrodynamic models and confinement effects by horizontal boundaries.
\newblock {\em Journal of Nonlinear Science}, 29(4):1445--1498, 2019.

\bibitem{castro2013breakdown}
{\'A}.~Castro, D.~C{\'o}rdoba, C.~Fefferman, and F.~Gancedo.
\newblock Breakdown of smoothness for the Muskat problem.
\newblock {\em Archive for Rational Mechanics and Analysis}, 208(3):805--909,
  2013.

\bibitem{constantin2018singularity}
P.~Constantin, T.~Elgindi, H.~Nguyen, and V.~Vicol.
\newblock On singularity formation in a Hele-Shaw model.
\newblock {\em Communications in Mathematical Physics}, 363(1):139--171, 2018.

\bibitem{cordoba2010interface}
A.~C{\'o}rdoba, D.~C{\'o}rdoba, and F.~Gancedo.
\newblock Interface evolution: water waves in 2-d.
\newblock {\em Advances in Mathematics}, 223(1):120--173, 2010.

\bibitem{cordoba2011interface}
A.~C{\'o}rdoba, D.~C{\'o}rdoba, and F.~Gancedo.
\newblock Interface evolution: the Hele-Shaw and Muskat problems.
\newblock {\em Annals of mathematics}, pages 477--542, 2011.

\bibitem{cordoba2007contour}
D.~C{\'o}rdoba and F.~Gancedo.
\newblock Contour dynamics of incompressible 3-d fluids in a porous medium with
  different densities.
\newblock {\em Communications in Mathematical Physics}, 273(2):445--471, 2007.

\bibitem{cordoba2014confined}
D.~C{\'o}rdoba~Gazolaz, R.~Granero-Belinch{\'o}n, and R.~Orive-Illera.
\newblock The confined Muskat problem: Differences with the deep water regime.
\newblock {\em Communications in Mathematical Sciences}, 12(3):423--455, 2014.

\bibitem{coutand2019finite}
D.~Coutand.
\newblock Finite-time singularity formation for incompressible Euler moving
  interfaces in the plane.
\newblock {\em Archive for Rational Mechanics and Analysis}, 232(1):337--387,
  2019.

\bibitem{CoSh2016}
D.~Coutand and S.~Shkoller.
\newblock On the impossibility of finite-time splash singularities for vortex sheets.
\newblock {\em Arch. Rational Mech. Anal.},     221:987--1033, 2016.

\bibitem{eggers1997nonlinear}
J.~Eggers.
\newblock Nonlinear dynamics and breakup of free-surface flows.
\newblock {\em Reviews of modern physics}, 69(3):865, 1997.

\bibitem{FeIoLi2016}
C. Fefferman, A.D. Ionescu,  and V. Lie, 
\newblock On the absence of ``splash'' singularities in the case of two-fluid interfaces.
\newblock {\em Duke Math. J.} 165:417?462, 2016.

\bibitem{gancedo2020surface}
F.~Gancedo, R.~Granero-Belinch{\'o}n, and S.~Scrobogna.
\newblock Surface tension stabilization of the Rayleigh-Taylor instability for
  a fluid layer in a porous medium.
\newblock 37(6):1299--1343, 2020.

\bibitem{gancedo2014absence}
F.~Gancedo and R.~M. Strain.
\newblock Absence of splash singularities for surface quasi-geostrophic sharp
  fronts and the Muskat problem.
\newblock {\em Proceedings of the National Academy of Sciences},
  111(2):635--639, 2014.

\bibitem{granero2020growth}
R.~Granero-Belinch{\'o}n and O.~Lazar.
\newblock Growth in the Muskat problem.
\newblock {\em Mathematical Modelling of Natural Phenomena}, 15:7, 2020.

\bibitem{lannes2005well}
D.~Lannes.
\newblock Well-posedness of the water-waves equations.
\newblock {\em Journal of the American Mathematical Society}, 18(3):605--654,
  2005.

\bibitem{lannes2013water}
D.~Lannes.
\newblock {\em The water waves problem: mathematical analysis and asymptotics},
  volume 188.
\newblock American Mathematical Soc., 2013.

\bibitem{liu2019local}
J.-G. Liu and R.~L. Pego.
\newblock On local singularities in ideal potential flows with free surface.
\newblock {\em Chinese Annals of Mathematics, Series B}, 40(6):925--948, 2019.

\bibitem{liu2021search}
J.-G. Liu and R.~L. Pego.
\newblock In search of local singularities in ideal potential flows with free
  surface.
\newblock {\em arXiv preprint arXiv:2108.00445}, 2021.

\bibitem{mariotte1700traite}
E.~Mariotte.
\newblock {\em Trait{\'e} de mouvement des eaux et des autres corps fluides...
  Mis en lumiere par les soins de M. de La Hire... Nouvelle {\'e}dition
  corrig{\'e}e}.
\newblock Jean Jombert, 1700.

\bibitem{moseler2000formation}
M.~Moseler and U.~Landman.
\newblock Formation, stability, and breakup of nanojets.
\newblock {\em Science}, 289(5482):1165--1169, 2000.

\bibitem{oron1997long}
A.~Oron, S.~H. Davis, and S.~G. Bankoff.
\newblock Long-scale evolution of thin liquid films.
\newblock {\em Reviews of modern physics}, 69(3):931, 1997.

\end{thebibliography}

\end{document}